\documentclass[12pt]{amsart}
\usepackage{amsmath}
\usepackage{amsfonts}
\begin{document}
\numberwithin{equation}{section}

\def\1#1{\overline{#1}}
\def\2#1{\widetilde{#1}}
\def\3#1{\widehat{#1}}
\def\4#1{\mathbb{#1}}
\def\5#1{\frak{#1}}
\def\6#1{{\mathcal{#1}}}

\newcommand{\de}{\partial}
\newcommand{\R}{\mathbb R}
\newcommand{\al}{\alpha}
\newcommand{\tr}{\widetilde{\rho}}
\newcommand{\tz}{\widetilde{\zeta}}
\newcommand{\tk}{\widetilde{C}}
\newcommand{\tv}{\widetilde{\varphi}}
\newcommand{\hv}{\hat{\varphi}}
\newcommand{\tu}{\tilde{u}}
\newcommand{\tF}{\tilde{F}}
\newcommand{\debar}{\overline{\de}}
\newcommand{\Z}{\mathbb Z}
\newcommand{\C}{\mathbb C}
\newcommand{\Po}{\mathbb P}
\newcommand{\zbar}{\overline{z}}
\newcommand{\G}{\mathcal{G}}
\newcommand{\So}{\mathcal{S}}
\newcommand{\Ko}{\mathcal{K}}
\newcommand{\U}{\mathcal{U}}
\newcommand{\B}{\mathbb B}
\newcommand{\oB}{\overline{\mathbb B}}
\newcommand{\Cur}{\mathcal D}
\newcommand{\Dis}{\mathcal Dis}
\newcommand{\Levi}{\mathcal L}
\newcommand{\SP}{\mathcal SP}
\newcommand{\Sp}{\mathcal Q}
\newcommand{\A}{\mathcal O^{k+\alpha}(\overline{\mathbb D},\C^n)}
\newcommand{\CA}{\mathcal C^{k+\alpha}(\de{\mathbb D},\C^n)}
\newcommand{\Ma}{\mathcal M}
\newcommand{\Ac}{\mathcal O^{k+\alpha}(\overline{\mathbb D},\C^{n}\times\C^{n-1})}
\newcommand{\Acc}{\mathcal O^{k-1+\alpha}(\overline{\mathbb D},\C)}
\newcommand{\Acr}{\mathcal O^{k+\alpha}(\overline{\mathbb D},\R^{n})}
\newcommand{\Co}{\mathcal C}
\newcommand{\Hol}{{\sf Hol}(\mathbb H, \mathbb C)}
\newcommand{\Aut}{{\sf Aut}(\mathbb D)}
\newcommand{\D}{\mathbb D}
\newcommand{\oD}{\overline{\mathbb D}}
\newcommand{\oX}{\overline{X}}
\newcommand{\loc}{L^1_{\rm{loc}}}
\newcommand{\la}{\langle}
\newcommand{\ra}{\rangle}
\newcommand{\thh}{\tilde{h}}
\newcommand{\N}{\mathbb N}
\newcommand{\kd}{\kappa_D}
\newcommand{\Hr}{\mathbb H}
\newcommand{\ps}{{\sf Psh}}
\newcommand{\Hess}{{\sf Hess}}
\newcommand{\subh}{{\sf subh}}
\newcommand{\harm}{{\sf harm}}
\newcommand{\ph}{{\sf Ph}}
\newcommand{\tl}{\tilde{\lambda}}
\newcommand{\gdot}{\stackrel{\cdot}{g}}
\newcommand{\gddot}{\stackrel{\cdot\cdot}{g}}
\newcommand{\fdot}{\stackrel{\cdot}{f}}
\newcommand{\fddot}{\stackrel{\cdot\cdot}{f}}
\def\v{\varphi}
\def\Re{{\sf Re}\,}
\def\Im{{\sf Im}\,}

%tentative title
\title[Evolution families II: hyperbolic manifolds]{Evolution Families and the Loewner Equation II: complex hyperbolic manifolds.}
\author[F. Bracci]{Filippo Bracci}
\address{F. Bracci: Dipartimento Di Matematica\\
Universit\`{a} di Roma \textquotedblleft Tor Vergata\textquotedblright\ \\
Via Della Ricerca Scientifica 1, 00133 \\
Roma, Italy} \email{fbracci@mat.uniroma2.it}
\thanks{$^\dag$Partially supported by the \textit{Ministerio
de Ciencia e Innovaci\'on} and the European Union (FEDER),
project MTM2006-14449-C02-01, by \textit{La Consejer\'{\i}a de
Educaci\'{o}n y Ciencia de la Junta de Andaluc\'{\i}a}, and by
the European Science Foundation Research Networking Programme
HCAA}
\author[M.D. Contreras]{Manuel D. Contreras$^\dag$ }
\address{M.D. Contreras \and  S. D\'{\i}az-Madrigal: Camino de los Descubrimientos, s/n\\
Departamento de Matem\'{a}tica Aplicada II \\
Escuela T\'ecnica Superior de Ingenieros\\
Universidad de Sevilla\\
41092, Sevilla\\
Spain.} \email{contreras@esi.us.es, madrigal@us.es}
\author[S. D\'{\i}az-Madrigal]{Santiago D\'{\i}az-Madrigal$^\dag$}

\date\today



\begin{abstract}
We prove that  evolution families on complex complete
hyperbolic manifolds  are in one to one correspondence with
certain semicomplete non-autonomous holomorphic vector fields,
providing the solution to a very general  Loewner type
differential equation on manifolds.
\end{abstract}

\maketitle

%\tableofcontents

%\def\Label#1{\label{#1}{\bf (#1)}~}
\def\Label#1{\label{#1}}

% Standard sets

\def\cn{{\C^n}}
\def\cnn{{\C^{n'}}}
\def\ocn{\2{\C^n}}
\def\ocnn{\2{\C^{n'}}}
\def\je{{\6J}}
\def\jep{{\6J}_{p,p'}}
\def\th{\tilde{h}}

% Abbreviations

\def\dist{{\rm dist}}
\def\const{{\rm const}}
\def\rk{{\rm rank\,}}
\def\id{{\sf id}}
\def\aut{{\sf aut}}
\def\Aut{{\sf Aut}}
\def\CR{{\rm CR}}
\def\GL{{\sf GL}}
\def\Re{{\sf Re}\,}
\def\Im{{\sf Im}\,}
\def\U{{\sf U}}

\def\la{\langle}
\def\ra{\rangle}

\emergencystretch15pt \frenchspacing

\newtheorem{theorem}{Theorem}[section]
\newtheorem{lemma}[theorem]{Lemma}
\newtheorem{proposition}[theorem]{Proposition}
\newtheorem{corollary}[theorem]{Corollary}

\theoremstyle{definition}
\newtheorem{definition}[theorem]{Definition}
\newtheorem{example}[theorem]{Example}

\theoremstyle{remark}
\newtheorem{remark}[theorem]{Remark}
\numberwithin{equation}{section}

\section{Introduction}

In \cite{Loewner} Loewner developed a tool to embed univalent
functions into particular families  of univalent functions,
nowadays known as {\sl Loewner chains}. Such a tool has been
studied and extended by many mathematicians in the past years
and has been proved to be very effective in the solution of
various problems and conjectures. For definitions, detailed
historical discussion and applications, we refer the reader to
the books of Pommerenke \cite{Pommerenke} and Graham and Kohr
\cite{Gra-Ko-book} and to \cite{BCM2}.

Loewner chains in the unit disc are related to other two
objects: a certain type of non-autonomous holomorphic vector
fields (which we call {\sl Herglotz vector fields}) and
families of holomorphic self-maps of the unit disc,  called
{\sl evolution families}. Those three objects are related by
means of differential equations, known as {\sl Loewner
differential equations}.

In one dimension, most of the work done so far is to relate
Loewner chains to Herglotz vector fields and to evolution
families when common fixed points inside the disc are present.
The main difficulty to treat boundary fixed points and higher
dimensional cases is the lack of appropriated distortion
theorems. For instance, we present here an apparently unknown
example of the lack of a type of Koebe $1/4$-theorem in higher
dimension (the first named author thanks Francois Berteloot for
discussions about it):

\begin{example}
Let $n\geq 2$ and let $F:\C^n\to \C^n$ be a holomorphic
injective map such that $F(\C^n)$ is a proper subset of $\C^n$
(this exists by the well known {\sl Fatou-Bieberbach
phenomenon}, see, \cite{Rudin-Rosay}). Up to translation we can
assume that $F(O)=O$. For $\lambda>0$ let
$F_\lambda(z):=\lambda F(z)$ and
$\Omega_\lambda=F_\lambda(\C^n)$. Then, for all fixed
$\epsilon>0$ there exists $\lambda>0$ such that
$\hbox{dist}(\de \Omega_\lambda, O)<\epsilon$. Now, let $A$ be
an $n\times n$ invertible matrix such that $d
(F_{\lambda}(Az))_{z=O}=\id$ and let $G(z):=F_\lambda(Az)$. The
map $G:\B^n\to \C^n$ is univalent, $G(O)=O$, $dG_O=\id$ and the
image of $G(\B^n)$ does not contain the ball of radius
$\epsilon$ and centered at $O$.
\end{example}

As a matter of fact, in higher dimension or with boundary fixed
points in the unit disc, in order to relate the three objects
and to solve the appropriated Loewner differential equations,
one is forced to restrict to particular classes of mappings.

Loewner differential equations have been studied in higher and
infinite dimension (in the unit ball) especially by Graham, G.
Kohr, M. Kohr, H. Hamada, T. Poreda and J.A. Pfaltzgraff. We
refer the reader to \cite{Gra-Ko-book}, \cite{GKK}, \cite{GHKK}
and references therein.

In this paper, which is a sequel of \cite{BCM2}, we study on
complex manifolds the Loewner differential equation which
relates evolution families and Herglotz vector fields.

To state our main result we need a few definitions. Let $M$ be
a complex manifold and let $d(\cdot, \cdot)$ be a distance on
$M$.

\begin{definition}\label{evolution}
A family $(\v_{s,t})$ is called a  {\sl (continuous) evolution
family} if
\begin{enumerate}
\item $\v_{s,s}={\sf id}_M$.
  \item $\v_{s,t}=\v_{u,t}\circ \v_{s,u}$ for all $0\leq s\leq
  u\leq t<+\infty$.
  \item $\v_{s,t}:M\to M$ is holomorphic for all $0\leq s\leq
  t<+\infty$.
  \item For any $0\leq s<+\infty$ and for any compact set
  $K\subset\subset M$ the function $[s,+\infty)\ni t\mapsto
  \v_{s,t}(z)$ is locally Lipschitz continuous  uniformly with respect to $K$.
  Namely, fixed $T>0$,  there exists $c_{T,K}>0$ such
  that
  \begin{equation}\label{ck-ev}
\sup_{z\in K}d(\v_{s,t}(z), \v_{s,t'}(z))<c_{T,K}|t-t'|,
  \end{equation}
  for all  $0\leq s\leq t,t'\leq T$.
\end{enumerate}
\end{definition}

A {\sl Herglotz vector field} $G(z,t)$ on $M$ is a weak
holomorphic vector field of order $\infty$ (see Section
\ref{vector}) such that $(dk_M)_{(z,w)}(G(z,t),G(w,t))\leq 0$
for almost every $t\in [0,+\infty)$ and all $z\neq w$. Here
$k_M$ is the Kobayashi distance on $M$ and it is assumed to be
$C^1$ outside the diagonal, which is always the case for
instance if $M$ is a strongly convex domain with smooth
boundary.

Our main result is that there is a one-to-one correspondence
between  evolution families and  Herglotz vector fields, namely

\begin{theorem}\label{main}
Let $M$ be a complete hyperbolic manifold with Kobayashi
distance $k_M$. Assume that $k_M\in C^1(M\times M\setminus
\hbox{Diag})$. Then for any  Herglotz vector field $G$ there
exists a unique evolution family $(\v_{s,t})$ over $M$ such
that for all $z\in M$
  \begin{equation}\label{solve}
\frac{\de \v_{s,t}(z)}{\de t}=G(\v_{s,t}(z),t) \quad
\hbox{a.e.\ } t\in [s,+\infty).
  \end{equation}
Conversely for any  evolution family $(\v_{s,t})$ over $M$
there exists a  Herglotz vector field $G$ such that
\eqref{solve} is satisfied. Moreover, if $H$ is another weak
holomorphic vector field which satisfies \eqref{solve} then
$G(z,t)=H(z,t)$ for all $z\in M$ and almost every $t\in
[0,+\infty)$.
\end{theorem}

Such a theorem is not sharp in the sense that some of the
hypotheses can be lowered in order to obtain partial results.
More precisely we show in Proposition \ref{her-to-ev} that
given a Herglotz vector field of order $d\geq 1$ there exists a
unique evolution family of the same order $d$ which verifies
\eqref{solve} (see Section \ref{vector} for definitions). The
converse (for $L^\infty$ data) is proved in Proposition
\ref{ev-to-her}, which holds more generally for taut manifolds.
Some remarks about the possibility of dropping the regularity
of $k_M$ are also contained at the end of Section \ref{proof}
and Section \ref{proof2}. We also show that all the elements of
an evolution family on taut manifolds must be univalent (see
Proposition \ref{univalent}).

\section{Weak holomorphic vector fields, Herglotz vector
fields and evolution families}\label{vector}

Let $M$ be a complex manifold of complex dimension $n$. Let
$TM$ denote the complex tangent bundle of $M$ and let $||\cdot
||$ be a Hermitian metric along the fibers of $TM$. We denote
by $d(\cdot, \cdot)$ the distance induced on $M$ by $||\cdot
||$. Also, we let $k_M$ denote the Kobayashi pseudo-distance on
$M$. For definitions and properties of $k_M$, of taut manifolds
and complete hyperbolic manifolds we refer the reader to, {\sl
e.g.}, \cite{Abate} or \cite{Kob}. Here we only recall that
complete hyperbolic manifolds are taut.

\begin{definition}
A {\sl weak holomorphic vector field of order $d\geq 1$} on $M$
is a function
\[
G:M\times [0,+\infty) \to TM
\]
with the following properties:
\begin{enumerate}
\item For all $z\in M$ the function $[0,+\infty)\ni t\mapsto
G(z,t)$ is measurable.
\item For all $t\geq 0$ the function $M\ni z\mapsto
G(z,t)$ is holomorphic.
\item For all compact set $K\subset\subset M$ and all $T>0$
there exists a function $C_{K,T}\in L^d([0,T],\R^+)$ such that
\begin{equation}\label{boundL1}
    ||G(z,t)||\leq C_{K,T}(t)
\end{equation}
for all $z\in K$ and almost every $t\in [0,T]$.
\end{enumerate}
\end{definition}

\begin{lemma}\label{lipsch}
Let $G(z,t)$ be a weak holomorphic vector field of order $d\geq
1$ on $M$. Let $U$ be a coordinate open set of $M$ such that
$TM|_U\simeq U\times \C^n$. Let $P\subset\subset U$ be a
relatively compact polydisc. Let $T>0$. Then there exists
$\tk_{P,T}\in L^d([0,T],\R^+)$ such that
\begin{equation}\label{lipL1}
    | G(z,t)-G(w,t) |\leq \tk_{P,T}(t) |z-w|
\end{equation}
for all $z,w\in P$ and almost every $t\in [0,T]$ (here
$|\cdot,\cdot |$ denotes the usual Hermitian metric on $\C^n$).
\end{lemma}

\begin{proof}
Let $(r_1,\ldots, r_n)$ be the multiradius of $P$ and let
$(\tilde{r}_1,\ldots, \tilde{r}_n)$ be the multiradius of
another polydisc $\tilde{P}$ such that $P\subset\subset
\tilde{P}\subset\subset U$. We write
\begin{equation*}
\begin{split}
|G(z,t)-G(w,t)|&\leq |G(z_1,\ldots,
z_n,t)-G(z_1,\ldots,z_{n-1},
w_n,t)|\\&+\ldots+|G(z_1,w_2\ldots, w_n,t)-G(w_1,\ldots,
w_n,t)|.
\end{split}
\end{equation*}
 By the Cauchy formula and
\eqref{boundL1} (taking into account that the Hermitian metric
$||\cdot ||$ is equivalent to the metric $| \cdot |$ of $\C^n$
on $U$) we have
\begin{equation*}
\begin{split}
|G(z_1,\ldots, z_n,t)-G(z_1,\ldots,z_{n-1}, w_n,t)|&\leq
\frac{1}{2\pi}\int_{|\xi|=\tilde{r}_n}\frac{|G(z_1,\ldots,z_{n-1},\xi,t)||z_n-w_n|}{|\xi-z_n||\xi-w_n|}|d\xi|\\
&\leq c_P C_{P,T}(t)|z_n-w_n|,
\end{split}
\end{equation*}
for some constant $c_P>0$ which depends only on $U$ and $P$.
Similar estimates hold for the other terms and thus
\eqref{lipL1} follows.
\end{proof}

\begin{remark}\label{escape}
By the Carath\'eodory theory of ODE's (see, for instance,
\cite{Coddington-Levison}) it follows that if $G(z,t)$ is a
weak holomorphic vector field on $M$, for any $(s,z)\in
[0,+\infty)\times M$ there exist  a unique $I(s,z)>s$ and a
function $x:[s, I(s,z))\to M$ such that
\begin{enumerate}
  \item $x$ is locally absolutely continuous in $t$.
  \item $x$ is the maximal solution to the following problem:
  \[
\begin{cases}
\frac{d x}{d t}(t):=x_\ast(\frac{\de}{\de t})=G(x(t),t) \quad
\hbox{ for a.e.\ }
t\in [s, I(s,z)),\\
x(s)=z.
\end{cases}
  \]
\end{enumerate}
The number $I(s,z)$ is referred to as the {\sl escaping time}
of the couple $(s,z)$.
\end{remark}

\begin{remark}\label{escapekob}
If $M$ is complete hyperbolic with Kobayashi metric $k_M$, $G$
is a weak holomorphic vector field over $M$ and $x$ is the
maximal solution for the couple $(s,z)$ with escaping time
$I(s,z)>0$ as defined in the previous remark, it follows that
for any compact set $K\subset\subset M$
\[
\limsup_{t\to I(s,z)}k_M(x(t),K)=+\infty.
\]
\end{remark}

\begin{definition}
Let $d\geq 1$. Let $M$ be a complex manifold. Assume that
$k_M\in C^1(M\times M\setminus \hbox{Diag})$. We let
\begin{equation*}
\begin{split}
\mathcal H_d(M):&=\{G(z,t) \hbox{ weak holomorphic vector field
of order $d$ }\\&: (dk_M)_{(z,w)}(G(z,t),G(w,t))\leq 0\quad
\forall z,w\in M, z\neq w \hbox{ and a.e. }t\in [0,+\infty)\}.
\end{split}
\end{equation*}
We call a {\sl Herglotz vector field of order $d$} any element
$G\in\mathcal H_d(M)$. A Herglotz vector field of order
$\infty$ is simply said a {\sl  Herglotz vector field}.
\end{definition}

\begin{remark}
According to \cite{BCD} if $M=D$ is a strongly convex domain in
$\C^n$ with smooth boundary then $G$ is a Herglotz vector field
if and only if it is a weak holomorphic vector field such that
for almost every $t\in [0,+\infty)$ the function $D\ni z\mapsto
G(z,t)$ is an infinitesimal generator of a semigroup of
holomorphic self-maps of $D$ (see, {\sl e.g.} \cite{Abate} for
definitions and properties of semigroups).
\end{remark}

\begin{definition}\label{evolutiond}
A family $(\v_{s,t})$ is called an  {\sl
 evolution family of order $d\geq 1$} if
\begin{enumerate}
\item $\v_{s,s}={\sf id}_M$.
  \item $\v_{s,t}=\v_{u,t}\circ \v_{s,u}$ for all $0\leq s\leq
  u\leq t<+\infty$.
  \item $\v_{s,t}:M\to M$ is holomorphic for all $0\leq s\leq
  t<+\infty$.
  \item For any $T>0$ and for any compact set
  $K\subset\subset M$ there exists a  function
  $c_{T,K}\in L^d([0,T],\R^+)$ such that
  \begin{equation}\label{ck-evd}
\sup_{z\in K}d(\v_{s,t}(z), \v_{s,t'}(z))\leq \int_{t'}^t
c_{T,K}(\xi)d \xi,
  \end{equation}
  for all  $0\leq s\leq t'\leq t\leq T$.
\end{enumerate}
An evolution family of order $\infty$ is simply said an {\sl
evolution family}.
\end{definition}

Since on a taut manifold $M$ the topology of pointwise
convergence on $\mathrm{Hol}(M,M)$ coincides with the topology
of uniform convergence on compacta (see, {\sl e.g.},
\cite[Corollary 2.1.17]{Abate}),  the proof of the following
lemma is similar to that of \cite[Proposition 3.4]{BCM2} and we
omit it:

\begin{lemma}\label{dx-cont}
Let $M$ be a taut complex manifold. Let $(\varphi_{s,t})$ be an
evolution family of order $d\geq 1$ in $M$. The map
$(s,t)\mapsto\varphi_{s,t}\in\mathrm{Hol}(M ,M)$ is jointly
continuous. Namely, given a compact set $K\subset M$ and two
sequences $\{s_{n}\},$ $\{t_{n}\}$ in $[0,+\infty),$ with
$0\leq s_{n}\leq t_{n},$ $s_{n}\rightarrow s,$ and
$t_{n}\rightarrow t$, then $\lim_{n\to
\infty}\varphi_{s_{n},t_{n}}=\varphi_{s,t}$ uniformly on $K.$
\end{lemma}

\section{From Herglotz vector fields to evolution families}\label{proof}

In this section we prove that to each Herglotz vector field of
order $d\geq 1$ there corresponds a unique evolution family of
order $d\geq 1$.

\begin{proposition}\label{her-to-ev}
Let $M$ be a complete hyperbolic manifold with Kobayashi
distance $k_M$. Assume that $k_M\in C^1(M\times M\setminus
\hbox{Diag})$. Let $d\geq 1$ and let $G(z,t)$ be a Herglotz
vector field of order $d\geq 1$ on $M$. Then there exists a
unique evolution family  $(\v_{s,t})$ of order $d$ on $M$ which
verifies \eqref{solve}.
\end{proposition}

In the proof, we will use the well-known Gronwall  Lemma. For
the reader convenience we state it here  as needed for our
aims.

\begin{lemma}
\label{Gronwall} Let $\theta:[a,b]\rightarrow\mathbb{R}$ be a
continuous function and $k\in L^{1}([a,b],\mathbb{R)}$
non-negative. If there exists $C\geq0$ such that for all
$t\in\lbrack a,b]$
\[
\theta(t)\leq C+\int_{a}^{t}\theta(\xi)k(\xi)d\xi\qquad\text{ (resp., }%
\theta(t)\leq C+\int_{t}^{b}\theta(\xi)k(\xi)d\xi),
\]
then%
\[
\theta(t)\leq C\exp\left(  \int_{a}^{t}k(\xi)d\xi\right)
\qquad\text{ (resp., }\theta(t)\leq C\exp\left(
\int_{t}^{b}k(\xi)d\xi\right)  \text{)}.
\]
\end{lemma}

\begin{proof}[Proof of Proposition \ref{her-to-ev}]

For any $s\in [0,+\infty)$ and $z\in M$, let $\v_{s,t}(z)$ be
the maximal solution with escaping time $I(s,z)>s$ (see Remark
\ref{escape}) which satisfies
\[
\begin{cases}
\frac{\de\v_{s,t}}{\de t}=G(\v_{s,t}(z),t) \quad \hbox{ for
a.e.\ }
t\in [s, I(s,z)),\\
\v_{s,s}(z)=z.
\end{cases}
\]
For all $s\geq 0$ and $z\in M$ the curve $[s,I(s,z))\ni
t\mapsto \v_{s,t}(z)$ is locally absolutely continuous.

{\bf 1.} {\sl For all $z,w\in M$ it holds  $I(0,z)=I(0,w)$}.
Indeed, fix $z,w\in M$. Then the function $h:[0,+\infty)\ni
t\mapsto k_M(\v_{0,t}(z), \v_{0,t}(w))$ is locally absolutely
continuous. Differentiating we get
\begin{equation*}
\begin{split}
\overset{\bullet}{h}(t)&=\frac{\de}{\de t}k_M(\v_{0,t}(z),
\v_{0,t}(w))\\&=(dk_M)_{(\v_{0,t}(z),
\v_{0,t}(w))}(G(\v_{0,t}(z),t),(\v_{0,t}(w),t))\leq 0\quad
\hbox{a.e.\ }t\in[0,+\infty).
\end{split}
\end{equation*}
Hence $h$ is decreasing in $t$ and therefore $h(t)\leq h(0)$
for all $t\in[0,+\infty)$. If $I(0,z)<I(0,w)$, since $k_M$ is
complete hyperbolic and
$\{\v_{0,t}(w)\}_{t\in[0,I(0,z)]}\subset\subset M$, by Remark
\ref{escapekob}, we have
\[
+\infty = \limsup_{t\to I(0,z)}k_M(\v_{0,t}(z),
\v_{0,t}(w))\leq k_M(z, w)<+\infty,
\]
a contradiction.  Then $I(0,z)\geq I(0,w)$. A similar argument
shows that $I(0,z)\leq I(0,w)$, thus Step 1 follows.

We let  $I:=I(0,z)$ for some $z\in M$.

{\bf 2.} {\sl For all $s<I$ and all $z\in M$ it follows
$I(s,z)=I$.} Indeed, repeating the argument in Step 1 with $s$
substituting $0$, we obtain that for all $z,w\in M$ and $0\leq
s\leq t<\min\{I(s,z), I(s,w)\}$ it follows
\begin{equation}\label{dis}
k_M(\v_{s,t}(z),\v_{s,t}(w))\leq k_M(z,w).
\end{equation}
Now, let $w=\v_{0,s}(z)$. Notice that $w$ is well defined if
$s<I$. By uniqueness of solutions of ODE's it follows that
$\v_{s,t}(w)=\v_{s,t}(\v_{0,s}(z))=\v_{0,t}(z)$. Therefore,
from \eqref{dis} it follows
\[
k_M(\v_{s,t}(z),\v_{0,t}(z))\leq k_M(z,\v_{0,s}(z)).
\]
Arguing as in Step 1 we obtain that $I(s,z)=I$ for all $s<I$.

{\bf 3.} {\sl $I=+\infty$}  (and thus $I(s,z)=+\infty$ for all
couple $(s,z)\in [0,+\infty)\times M$ by Step 2).

Fix $z_0\in M$. We are going to show that there exists
$\delta>0$ such that for all $s\in [0, I)$ it follows
$I(s,z)\geq s+\delta$. If this is true and $I<+\infty$, then
letting $I-\delta<s<I$ it follows that $I(s,z)>I$ contradicting
Step 2, and thus $I=+\infty$.

To prove the existence of $\delta>0$ as before, let $U$ be a
local chart of $M$ which trivializes $TM$ and such that $z_0$
has coordinates $O$. With no loss of generality we can assume
that $U$ contains  a closed polydisc $P$ which contains the
ball $\B$ of radius $1$ and center $O$ in $\C^n$. Let $r<1$.
Let $\B_r:=\{z\in \C^n: |z|\leq r\}$. By the very definition of
weak holomorphic vector field of order $d$ there exists
$C:=C_{\B_r,[0,I+2]}\in L^d([0,I+2],\R^+)$ such that
\begin{equation}\label{es-G}
    |G(z,t)|\leq C(t)
\end{equation}
for all $z\in \B_r$ and almost every $t\in [0,I+2]$. Moreover,
by Lemma \ref{lipsch}, there exists $\tk :=\tk_{P,[0,I+2]}\in
L^d([0,I+2],\R^+)$ such that
\begin{equation}\label{es-G2}
    |G(z,t)-G(w,t)|\leq \tk(t)|z-w|
\end{equation}
for all $z,w\in \B_r$ and almost every $t\in [0,I+2]$.

The functions $[0,I+2]\ni u\mapsto \int_0^{u}C(\tau)d\tau$,
$[0,I+2]\ni u\mapsto \int_0^{u}\tk(\tau)d\tau$ are absolutely
continuous and therefore there exists $\delta>0$ (which we can
suppose strictly less than $1$) such that for all $s\in
[0,I+1]$ it holds
\begin{equation}\label{stima1}
    \int_s^{s+\delta} C(\tau)d\tau \leq r, \quad \int_s^{s+\delta} \tk(\tau)d\tau \leq
    r.
\end{equation}
For $s\in [0,I+1]$ let us define by induction
\[
\begin{cases}
x^s_0(t):=O\\
x^s_n(t)=\int_s^t G(x_{n-1}^s(\tau),\tau)d\tau \quad t\in
[s,s+\delta].
\end{cases}
\]
We notice that $|x^s_0(t)|=0<r$. Assuming that
$|x^s_{n-1}(t)|\leq r$ for all $t\in [s,s+\delta]$,   by
\eqref{es-G} and \eqref{stima1} we have
\[
|x^s_n(t)|\leq \int_s^t |G(x_{n-1}^s(\tau),\tau)|d\tau\leq
\int_s^t C(\tau)d\tau \leq r, \quad t\in [s,s+\delta]
\]
which, by induction, implies that $x_n^s(t)$ is well defined
for all $n\in \N$ and $t\in [s,s+\delta]$ and $|x^s_n(t)|\leq
r$.

Now, by \eqref{es-G2} and \eqref{stima1} we have
\begin{equation*}
\begin{split}
|x_n^s(t)-x_{n-1}^s(t)|&\leq \int_s^t
|G(x_{n-1}^s(\tau),\tau)-G(x_{n-2}^s(\tau),\tau)|d\tau\\ &\leq
\int_s^t \tk(\tau)|x_{n-1}^s(\tau)-x_{n-2}^s(\tau)|d\tau
\\ & \leq  \max_{\tau\in [s,s+\delta]}
|x_{n-1}^s(\tau)-x_{n-2}^s(\tau)|\int_s^t \tk(\tau)d\tau,
\end{split}
\end{equation*}
for $n\in \N$ and $t\in [s,s+\delta]$. From this follows   that
$\{x_n^s\}$ is a Cauchy sequence in the Banach space
$C^0([s,s+\delta], \C^n)$. Therefore it converges uniformly on
$[s,s+\delta]$ to a function $x^s\in C^0([s,s+\delta], \B_r)$.
By \eqref{es-G} and the Lebesgue dominated converge theorem it
follows that
\[
x^s(t)=\int_s^t G(x^s(\tau),\tau)d\tau \quad \forall t\in
[s,s+\delta],
\]
or, in other words,
\[
\v_{s,t}(z_0)=x^s(t) \quad \forall t\in [s,s+\delta]
\]
which proves that $I(z_0,s)\geq s+\delta$ as needed.

{\bf 4.} {\sl $(\v_{s,t})$ is an ``algebraic'' evolution
family}, namely, Properties (1) and (2) of Definition
\ref{evolutiond} hold. This follows at once by the uniqueness
of solutions of ODE's.

{\bf 5.} {\sl For all fixed $0\leq s\leq t<+\infty$ the map
$M\ni z\mapsto \v_{s,t}(z)\in M$ is holomorphic}. Let $z_0\in
M$ and let $0\leq s\leq t<+\infty$. The  absolutely continuous
(compact) curve $[s,t]\ni \eta\mapsto \v_{\eta,t}(z_0)$ is
covered by a finite number of coordinates charts, so that we
can find a partition $s=t_0<t_1<\ldots<t_m=t$ such that each
curve $[t_j,t_{j+1}]\ni \eta\mapsto \v_{\eta,t_{j+1}}(z_0)$ is
contained in a coordinates chart. By Property (2) of Definition
\ref{evolutiond} (which holds for $(\v_{s,t})$ by Step 4),
holomorphicity of $\v_{s,t}$ at $z_0$ will follow as soon as we
can show holomorphicity of $\v_{t_{j-1},t_j}$ at
$\v_{s,t_j}(z_0)$ for $j=1,\ldots, m$. Therefore, we can
suppose that the curve $[s,t]\ni \eta\mapsto \v_{\eta,t}(z_0)$
is contained in a local chart.

By \eqref{dis}, for any open neighborhood $V$ of $z_0$,
relatively compact in $U$, there exists a open neighborhood
$W\subset U$ of $z_0$ such that for any $w\in W$ and all
$\eta\in [s,t]$ it follows that $\v_{s,\eta}(w)\in V$.

Since holomorphy is a local property, we can work on the local
chart $U$ (centered at $z_0$), which we may assume  is the ball
$\B:=\{z\in \C^n: |z|<1\}$ of center $O$ and radius $1$. In
such local coordinates we can write $G=(G_1,\ldots, G_n)$. For
$\eta\in [s,t]$ let
\[
A(\eta):=\left(
           \begin{array}{ccc}
             \frac{\de G_1(\v_{s,\eta}(O),\eta)}{\de z_1} & \ldots & \frac{\de G_1(\v_{s,\eta}(O),\eta)}{\de z_n} \\
             \vdots & \vdots & \vdots \\
             \frac{\de G_n(\v_{s,\eta}(O),\eta)}{\de z_1} & \ldots & \frac{\de G_n(\v_{s,\eta}(O),\eta)}{\de z_n} \\
           \end{array}
         \right).
\]
By Cauchy formula and \eqref{boundL1} (arguing similarly to the
proof of Lemma \ref{lipsch}), each entry of $A(\eta)$ is a
$L^d$-measurable function in $\eta$. It is well known that the
following system of ODE's
\[
\begin{cases}
\frac{d H}{d \eta}(\eta)=-H(\eta)\cdot A(\eta)\\
H(s)=\sf{Id}
\end{cases}
\]
has a unique continuous solution $H(\eta)$ which is an
invertible $n\times n$ matrix for all $\eta$.

Let $v\in \C^n$ be such that $|v|=1$. We will prove that
\begin{equation}\label{daprov1}
   \lim_{h\in \C,|h|\to
   0}\frac{\v_{s,t}(hv)-\v_{s,t}(0)}{h}=H^{-1}(t)\cdot v,
\end{equation}
showing that $z\mapsto \v_{s,t}(z)$ is holomorphic.

Since the topology induced by $k_M$ coincides with the one of
$M$, the previous argument based on \eqref{dis} shows that
there exists $\delta>0$ such that
\begin{equation}\label{sup-est}
\sup\{|\v_{s,\eta}(z)|: |z|< \delta, \eta\in [s,t]
\}<\frac{1}{5n}.
\end{equation}
Let $P:=\{z\in \B^n: \max_j |z_j|<1/(4n)\}$ be the polydisc of
center $O$ and multi-radius $1/(4n)$. Let $\tk_{P,t}\in
L^d([0,t])$ be as in  \eqref{lipL1} (with $T=t$). We let
\[
\theta(\eta):=|\v_{s,\eta}(hv)-\v_{s,\eta}(0)|,
\]
where $h\in \C$ is such that $|h|<\delta$ and $\eta\in [s,t]$.
Then $\theta\in C^0([s,t], \R^+)$. Since $\v_{s,\eta}(z)$
solves \eqref{solve}, and $\v_{s,\eta}(hv)\in P$ for all $\eta
\in [s,t]$ and $|h|<\delta$ by \eqref{sup-est}, it follows that
\begin{equation*}
\begin{split}
\theta(\eta)&=\left|hv+\int_s^\eta
G(\v_{s,\xi}(hv),\xi)d\xi-\int_s^\eta
G(\v_{s,\xi}(O),\xi)d\xi\right|\\ &\leq |h|+\int_s^\eta
\left|G(\v_{s,\xi}(hv),\xi)-G(\v_{s,\xi}(O),\xi)\right|\
d\xi\\&\leq \theta(s)+\int_s^\eta \tk_{P,t}(\xi)d\xi.
\end{split}
\end{equation*}
Gronwall's Lemma \ref{Gronwall} implies then
\[
\theta(\eta)\leq \theta(s)\exp\left(\int_s^\eta
\tk_{P,t}(\xi)d\xi \right)\quad \forall \eta\in [s,t].
\]
Therefore, setting $C(s,t,P):=\exp(\int_s^t
\tk_{P,t}(\xi)d\xi)$, we have
\begin{equation}\label{A-1}
    |\v_{s,\eta}(hv)-\v_{s,\eta}(0)|\leq |h|C(s,t,P) \quad \forall \eta\in [s,t].
\end{equation}
Fix $h$ and let
$f_h(\eta):=(\v_{s,\eta}(hv)-\v_{s,\eta}(O))/h$. Let us denote
$\v_{s,\eta}(z)=(\v_{s,\eta}^1(z),\ldots, \v_{s,\eta}^n(z))$.
In what follows we assume that $\v_{s,\eta}^j(hv)\neq
\v^j_{s,\eta}(O)$ for all $0<h<<1$, $j=1,\ldots, n$, and leave
to the reader the (obvious) changes in case
$\v_{s,\eta}^j(hv)\equiv \v^j_{s,\eta}(O)$. Then, for almost
every $\eta\in [s,t]$, it follows
\begin{equation*}
\begin{split}
\frac{d f_h}{d \eta}(\eta)&=\frac{\frac{\de
\v_{s,\eta}(hv)}{\de \eta}-\frac{\de \v_{s,\eta}(O)}{\de
\eta}}{h}=\frac{G(\v_{s,\eta}(hv),\eta)-G(\v_{s,\eta}(O),\eta)}{h}\\&=
\frac{(\v_{s,\eta}^1(hv)-\v^1_{s,\eta}(O))}{h}\frac{G(\v_{s,\eta}(hv),\eta)-G((\v_{s,\eta}^1(O),
\v_{s,\eta}^2(hv),\ldots,\v^n_{s,\eta}(O)),\eta)}{\v_{s,\eta}^1(hv)-\v^1_{s,\eta}(O)}\\
&+\ldots+\frac{(\v_{s,\eta}^n(hv)-\v^n_{s,\eta}(O))}{h}\frac{G((\v_{s,\eta}^1(O),\ldots,
\v_{s,\eta}^{n-1}(O),\v^n_{s,\eta}(hv)),\eta)-G(\v_{s,\eta}(O),\eta)}{\v_{s,\eta}^n(hv)-\v^n_{s,\eta}(O)}
\\ &=A(\eta) f_h(\eta) + \sum_{j=1}^n L^h_j(\eta),
\end{split}
\end{equation*}
where
\begin{equation*}
\begin{split}
L^h_1(\eta)=\frac{(\v_{s,\eta}^1(hv)-\v^1_{s,\eta}(O))}{h}&\left[\frac{G(\v_{s,\eta}(hv),\eta)-G((\v_{s,\eta}^1(O),
\v_{s,\eta}^2(hv),\ldots,\v^n_{s,\eta}(O)),\eta)}{\v_{s,\eta}^1(hv)-\v^1_{s,\eta}(O)}\right.\\&
\left.-\frac{\de G}{\de z_1}(\v_{s,\eta}(O),\eta) \right],
\end{split}
\end{equation*}
and $L^h_2(\eta),\ldots, L^h_n(\eta)$ are defined similarly.
Thus, multiplying by $H(\eta)$,
\[
H(\eta)\frac{d f_h}{d \eta}(\eta)-H(\eta)A(\eta) f_h(\eta) =
H(\eta)\sum_{j=1}^n L^h_j(\eta)\quad \hbox{a.e. } \eta \in
[s,t],
\]
and, by the very definition of $H$, it holds
\[
\frac{d}{d \eta}(H(\eta)f_h(\eta))=H(\eta)\sum_{j=1}^n
L^h_j(\eta)\quad \hbox{a.e. } \eta \in [s,t].
\]
Integrating  we obtain
\[
H(t)f_h(t)-f_h(s)=\sum_{j=1}^n\int_{s}^tH(\eta) L^h_j(\eta)d
\eta.
\]
But $f_h(s)=v$ for all $h$. Moreover \eqref{lipL1} and
\eqref{A-1} imply that there exists a constant $a>0$
(independent of $h$) such that  $|H(\eta) L^h_j(\eta)|\leq a
\tk_{P,t}(\eta)$ for all $h$. The Lebesgue dominated
convergence theorem implies then $\lim_{h\to
0}\int_{s}^tH(\eta) L^h_j(\eta)d \eta=0$. Hence
\[
\lim_{h\in \C, h\to 0}H(t)f_h(t)-v=0,
\]
proving \eqref{daprov1}.

{\bf 6.} {\sl for all fixed $T>0$ and $K\subset\subset M$
compact set, there exists a  function  $c_{T,K}\in
L^d([0,T],\R^+)$ such that \eqref{ck-evd} is satisfied
  for all  $0\leq s\leq t'\leq t\leq T$}. Fix $s\geq 0$ and $K\subset\subset M$. Let $T>0$ and assume
$s\leq t'\leq t\leq T$. Let $V$ be a compact set in $M$ which
contains $\{\v_{s,\eta}(z)\}$ for all $z\in K$ and $\eta\in
[s,T]$. Since the curve $\eta\mapsto \v_{s,\eta}(z)$ is
absolutely continuous, and the distance between two points can
be computed as the infimum of absolutely continuous curves
(see, {\sl e.g.}, \cite{Fed}),  by \eqref{boundL1}
\begin{equation*}
\begin{split}
d(\v_{s,t'}(z),\v_{s,t}(z))&\leq \int_{t'}^t \|\frac{\de
\v_{s,\eta}(z)}{\de t}\|d\eta=\int_{t'}^t
\|G(\v_{s,\eta}(z),\eta)\|d\eta\\&\leq \int_{t'}^t
C_{T,V}(\tau)d\tau,
\end{split}
\end{equation*}
and we are done.

Thus $\{\v_{s,t}(z)\}$ is an evolution family of order $d\geq
1$ over $M$. Uniqueness follows at once from the uniqueness of
solutions of ODE's.
\end{proof}

\begin{remark}
One can drop the hypothesis of regularity of $k_M$ from the
statement of Proposition \ref{her-to-ev} by giving a meaning in
the sense of currents and distributions to the inequality
$dk_M(G,G)\leq 0$, similarly to what has  been done by M. Abate
in \cite{Abatex} for infinitesimal generators of semigroups
using the Kobayashi infinitesimal metric. We leave the details
to the interested reader.
\end{remark}

\section{From  evolution families to
Herglotz vector fields}\label{proof2}

\begin{proposition}\label{ev-to-her}
Let $M$ be a taut manifold with Kobayashi distance $k_M$.
Assume that $k_M\in C^1(M\times M\setminus \hbox{Diag})$. Then
for any  evolution family $(\v_{s,t})$ over $M$ there exists a
 Herglotz vector field $G\in \mathcal H_\infty(M)$
which verifies \eqref{solve}.
 Moreover, if $H$
is another weak holomorphic vector field which satisfies
\eqref{solve} then $G(z,t)=H(z,t)$ for all $z\in M$ and almost
every $t\in [0,+\infty)$.
\end{proposition}

\begin{proof}
 We show that the Herglotz
vector field $G(z,t)$ is ``morally'' the vector field tangent
to the (locally absolutely) continuous curve $[0,+\infty)\ni
h\mapsto \v_{t,t+h}(z)$ at $h=0$.

First of all, we fix a local coordinates chart $U$. Thus, there
exists a biholomorphism $\Phi: U\to \Phi(U)\subset \C^n$. We
are going to define a holomorphic vector field which solves
\eqref{solve} on $\Phi(U)$ and then pull it back via $\Phi$.
With a customary abuse of notation, in the sequel we avoid
writing $\Phi$, for instance, we write $\v_{s,t}(z)$ instead
$\Phi\circ \v_{s,t}(z)$ and so on.

{\bf 1'.} {\sl Local definition of an approximation family}.
Let $U'\subset\subset U$ be a relatively compact open subset of
$U$ and let $T>0$. By Lemma \ref{dx-cont} there exists
$h_0=h_0(U',T)>0$ such that $\v_{t,t+h}(z)\in U$ for all $z\in
U'$, $t\in [0,T]$
 and $0\leq h\leq h_0$. We let
\[
G_h(z,t):=\frac{\v_{t,t+h}(z)-z}{h} \quad z\in U',\quad 0<h\leq
h_0,\quad t\in [0,T].
\]

{\bf 2'.} {\sl For every $0<h\leq h_0$ fixed, the map $U'\ni
z\mapsto G_h(z,t)$ is holomorphic   for all fixed $t$}. It
follows immediately from the very definition.

{\bf 3'.} {\sl The function $[0,T]\ni t\mapsto G_h(z,t)$ is
continuous}. It follows at once from Lemma \ref{dx-cont}.

{\bf 4'.} {\sl For every compact set $K\subset\subset U'$
there exists $A_{T,K}>0$ such that
\[
|G_h(z,t)|\leq A_{T,K}
\]
for all $0<h\leq h_0$, $z\in K$ and  every $t\in [0,T]$}.
Indeed, by Property (4) of Definition \ref{evolution} we have
\[
|\v_{t,t+h}(z)-z|\leq  c_U d(\v_{t,t+h}(z),\v_{t,t}(z))\leq c_U
c_{T,K}|t+h-t|=c_Uc_{T,K}|h|,
\]
and setting $A_{T,K}=c_Uc_{T,K}$ we have the claim.

{\bf 5'.} {\sl For all $t\in [0,+\infty)$ there exists a
sequence $m_j(t)\to 0^+$ such that $G(z,t)=\lim_{j\to
\infty}G_{m_j(t)}(z,t)$ has the property that $U'\ni z\mapsto
G(z,t)$ is holomorphic for all fixed $t$ and $[0,+\infty)\ni
t\mapsto G(z,t)$ is measurable for all fixed $z\in U'$.}  Let
$\{K_r\}$ be a compact exhaustion of $U'$ and let $T\in
\N\setminus\{0\}$. Let
\[
\Upsilon_T:=\{f:U'\to \C^n: \sup_{z\in K_r}|f(z)|\leq
A_{T,K_r}, r\in(0,+\infty) \},
\]
where $A_{T,K_r}>0$ is the constant given in Step 4'. By Step.
4' the family $\{G_h(z,t)\}$ is normal for every fixed $t\geq
0$, thus every accumulation point is holomorphic.

The space $\Upsilon_T$ is a closed bounded subset of ${\sf
Hol}(U,\C^n)$. Thus it is a compact metrizable space (with
respect to the topology of uniform convergence on compacta).

Let $\gamma_T: \Upsilon_T^{\mathbb N}\to \Upsilon_T$ be the
measurable selector (see, {\sl e.g.}, \cite{Lu}). Then consider
the sequence $\{G_{1/m}(z,t)\}_{m\in \N}\subset \Upsilon_T$ and
let
\begin{equation}\label{limitG}
G^T(z,t):=\gamma_T(\{G_{1/m}(z,t)\}):=\lim_{j\to
\infty}G_{\frac{1}{m_j(t)}}(z,t).
\end{equation}
 By the very construction and
definition of measurable selector, $[T-1,T)\ni t\mapsto
G^T(z,t)$ is measurable for all $z$. Then define
$G(z,t):=G^T(z,t)$ according to whether $t\in [T-1,T)$.

{\bf 6'.} {\sl The map $G(z,t)$ is a weak holomorphic vector
field of order $\infty$ over $U'$ and satisfies \eqref{solve}.}
By Step 5', since $G(z,t)\in \Upsilon_T$, it is clear that
$G(z,t)$ is a weak holomorphic vector field of order $\infty$
on $U$. Now, by \eqref{limitG} and for almost every $t\in
[0,+\infty)$ (and for $s\leq t$ so that $\v_{s,\eta}(z)\in U'$
for $\eta\in [s,t]$) we have
\begin{equation*}
\begin{split}
\frac{\de \v_{s,t}(z)}{\de t}&=\lim_{h\to
0}\frac{\v_{s,t+h}(z)-\v_{s,t}(z)}{h}=\lim_{h\to
0}\frac{\v_{t,t+h}(\v_{s,t}(z))-\v_{s,t}(z)}{h}\\&= \lim_{j\to
\infty}\frac{\v_{t,t+h_j(t)}(\v_{s,t}(z))-\v_{s,t}(z)}{h_j(t)}=G(\v_{s,t}(z),t).
\end{split}
\end{equation*}

{\bf 7'}. {\sl For every polydisc $P\subset\subset U'$ and for
all $0\leq s\leq t$ such that there exists a polydisc
$P'\subset\subset U$ for which $\v_{s,\eta}(P)\subseteq P'$ for
all $\eta\in [s,t]$, it follows that $\v_{s,t}:P\to U$ is
univalent}. Let $z,w\in P$ and $z\neq w$. Assume by
contradiction that $\v_{s,t}(z)=\v_{s,t}(w)$. Let
\[
\theta(\eta):=|\v_{s,\eta}(z)-\v_{s,\eta}(w)|, \quad \eta\in
[s,t].
\]
Let $\tk_{K,t}\in L^\infty([0,t])$ be as in \eqref{lipL1} (with
$T=t$). Then
\begin{equation*}
\begin{split}
\theta(\eta)&=|\v_{s,\eta}(z)-\v_{s,t}(z)-\v_{s,\eta}(w)+\v_{s,t}(w)|
=\left|\int_{\eta}^t
(G(\v_{s,\xi}(w),\xi)-(G(\v_{s,\xi}(z),\xi)d\xi\right|
\\&\leq \int_{\eta}^t
|G(\v_{s,\xi}(z),\xi)-G(\v_{s,\xi}(w),\xi)|d\xi\leq |
\tk_{K,t}|_{L^\infty}\int_\eta^t \theta(\xi)d\xi.
\end{split}
\end{equation*}
Gronwall's Lemma \ref{Gronwall} implies that $\theta(\eta)= 0$
for all $\eta\in [s,t]$. But $\theta(s)=|z-w|\neq 0$, reaching
a contradiction.

{\bf 8'.} {\sl Uniqueness of $G(z,t)$ on $U'$}. We are going to
prove that if $H(z,t)$ is another weak holomorphic vector field
over $U'$ such that
\begin{equation}\label{solve2}
\frac{\de \v_{s,t}(z)}{\de t}=H(\v_{s,t}(z),t)
\end{equation}
for almost every $t\in [0,+\infty)$ (and for $s\leq t$ so that
$\v_{s,\eta}(z)\in U'$ for $\eta\in [s,t]$) then
$H(z,t)=G(z,t)$ for all $z\in U'$ and almost every $t\in
[0,+\infty)$. Let $P\subset\subset U'$ be a polydisc and let
$s\approx t$ so that $\v_{s,\eta}(P)$ is contained in a fixed
polydisc relatively compact in $U'$ for all $\eta\in [s,t]$.
Then from \eqref{solve} and \eqref{solve2} we obtain that for
almost every $t\in [0,+\infty)$ and $q\in \v_{s,t}(P)$
\[
H(q,t)\equiv G(q,t).
\]
Since $z\mapsto \v_{s,t}(z)$ is  univalent on $P$ by Step 7',
 the set $\v_{s,t}(P)$ is open in $U'$ and  thus by the
identity principle for holomorphic maps $G(z,t)\equiv H(z,t)$
for all $z\in U'$ and almost every $t\in [0,+\infty)$.

Now  we have a way to define the vector field solving
\eqref{solve} on each relatively compact open set of any local
chart in $M$, and such a vector field is unique for almost
every time. Since $M$ is countable at infinity  by the very
definition of manifold, we can cover it with  countable many
coordinates charts in such a way that this covering is locally
finite. Hence the previous construction allows to define
globally $G(z,t)$ on $M$ for almost every $t\in [0,+\infty)$ in
such a way that \eqref{solve} is satisfied.

To end up the proof, we need to show that  $G(z,t)$ is a
Herglotz vector field. To this aim, it is only left to show
that $(dk_M)_{(z,w)}(G(z,t),G(w,t))\leq 0$ for a.e.
$t\in[0,+\infty)$, $z\neq w$. The map $M\ni
z\mapsto\v_{s,t}(z)$ is holomorphic, thus by Property (2) of
Definition \ref{evolution}
\begin{equation}\label{minoP}
\begin{split}
k_M(\v_{s,t+h}(z),\v_{s,t+h}(w))&\leq
k_M(\v_{t,t+h}(\v_{s,t}(z)),\v_{t,t+h}(\v_{s,t}(w)))\\&\leq
k_M(\v_{s,t}(z),\v_{s,t}(w))
\end{split}
\end{equation}
for all $h\geq 0$.

Let $Z(s)$ be the zero measure set such that
$[s,+\infty)\setminus Z(s)\ni t\mapsto \v_{s,t}(z)$ and
$[s,+\infty)\setminus Z(s)\ni t\mapsto \v_{s,t}(w)$ are
differentiable. Let $Z:=\cup_{s\in \mathbb Q} Z(s)$. Then $Z$
has zero  measure. Let $t_0\in [0,+\infty)\setminus Z$. By
\eqref{minoP} and \eqref{solve} (and taking into account that
$\v_{s,t}(z)\neq \v_{s,t}(w)$ as we will show in Proposition
\ref{univalent}), we have
\begin{equation*}
\begin{split}
 0&\geq \lim_{h\to
0}\frac{k_M(\v_{s,t_0+h}(z),\v_{s,t_0+h}(w))-k_M(\v_{s,t_0}(z),\v_{s,t_0}(w))}{h}
\\&=(dk_M)_{(\v_{s,t_0}(z),\v_{s,t_0}(w))}(G(\v_{s,t_0}(z),t_0),G(\v_{s,t_0}(w),t_0)),
\end{split}
\end{equation*}
which holds for every $s\in \mathbb Q$ such that $s\leq t_0$.
Taking the limit for $\mathbb Q\ni s\to t_0$ by
Lemma~\ref{dx-cont} we have the result.
\end{proof}

\begin{remark}
If $M=\D$ the unit disc in $\C$, in \cite{BCM2} we proved that
Proposition \ref{ev-to-her} holds also for evolution families
of order $d\geq 1$ (and not just $d=\infty$). This was done by
using the Berkson-Porta representation formula for
infinitesimal generators, which is missing in higher
dimensions.
\end{remark}

\begin{remark}
Dropping the hypothesis on the regularity of $k_M$ in
Proposition \ref{ev-to-her}, it follows from the above proof
that given an evolution family $(\v_{s,t})$ on $M$ then  there
exists a weak holomorphic vector field $G$ of order $\infty$
which verifies \eqref{solve}. It does not seem to be clear how
to obtain that $G$ is Herglotz, not even in the distributional
sense.
\end{remark}

\section{Univalence of evolution families}\label{Commenti}

\begin{proposition}\label{univalent}
Let $(\v_{s,t})$ be an evolution family on a taut manifold $M$.
Then for every $0\leq s\leq t<+\infty$ the map $M\ni z \mapsto
\v_{s,t}(z)\in M$ is univalent.
\end{proposition}

\begin{proof}
 Fix $0\leq s\leq t<+\infty$. We show that
$\v_{s,t}$ is injective on $M$. By contradiction, assume
$\v_{s,t}(z_0)=\v_{s,t}(z_1)$ with $z_0\neq z_1$. Let $t_0\in
[s,t]$ be the smallest number such that
$\v_{s,t_0}(z_0)=\v_{s,t_0}(z_1)$. Since $\v_{s,s}=\id_M$,
Lemma \ref{dx-cont} implies  $t_0>s$. Let
$z_2=\v_{s,t_0}(z_0)$. Then for all $u\in (s,t_0)$
\begin{equation}\label{pp1}
\v_{u,t_0}(\v_{s,u}(z_0))=\v_{s,t_0}(z_0)=z_2=\v_{s,t_0}(z_1)=\v_{u,t_0}(\v_{s,u}(z_1)).
\end{equation}
Let $U$ be a local chart around $z_2$ and let $P$ be a
relatively compact polydisc in $U$. By Lemma \ref{dx-cont} we
can choose  $s<u<t_0$   such that both $\v_{s,u}(z_0)$ and
$\v_{s,u}(z_1)$ are contained in $P$ and $\v_{u,\eta}(P)$ is
contained in a relatively compact polydisc in $U$ for all
$\eta\in [u,t_0]$. Looking at the proof of Proposition
\ref{ev-to-her} (Steps 1' to 7') it follows that $\v_{u,t_0}$
is injective on $P$. But $\v_{s,u}(z_0)\neq \v_{s,u}(z_1)$  by
definition of $t_0$, hence we get a contradiction with
\eqref{pp1}.
\end{proof}

\bibliographystyle{alpha}

\end{document}